\documentclass[11pt]{article}%
\usepackage{geometry}%
\usepackage{amsmath}%
\usepackage{amsfonts}%
\usepackage{amssymb}%
\usepackage{graphicx}
\usepackage{mathrsfs}
\usepackage{amsthm}
\usepackage{amscd}
\usepackage{mathpazo}
\usepackage{txfonts}

\usepackage{amsbsy}
\usepackage{enumerate}

\usepackage{mathtools}
\allowdisplaybreaks[1]
\usepackage{color}
\usepackage{lastpage}
\usepackage{fancyhdr}
\usepackage[T1]{fontenc}
\usepackage{graphicx}
\usepackage[dvipsnames]{xcolor}
\usepackage{centernot}
\allowdisplaybreaks[1]
\usepackage{tikz-cd}
\newtheorem{theorem}{Theorem}

\newtheorem{corollary}[theorem]{Corollary}

\newtheorem{proposition}[theorem]{Proposition}
\newtheorem{lemma}[theorem]{Lemma}

\theoremstyle{definition}
\newtheorem{definition}[theorem]{Definition}

\newtheorem{remark}[theorem]{Remark}

\newcommand{\assign}{:=}

\newcommand{\llangle}{{\langle\!\langle}}
\newcommand{\nobracket}{}
\newcommand{\rrangle}{{\rangle\!\rangle}}

\newcommand{\tmtextit}[1]{{\itshape{#1}}}

\newcommand{\WW}{\mathbb{W}}
\newcommand{\PP}{\mathbb{P}}
\def\P{{\mathbb P}}

\def\N{{\mathbb N}}

\def\eps{\varepsilon}

\newcommand{\ep}{\varepsilon}

\makeatletter
\@tfor\next:=abcdefghijklmnopqrstuvwxyzABCDEFGHIJKLMNOPQRSTUVWXYZ\do{%
\def\command@factory#1{%
\expandafter\def\csname b#1\endcsname{\mathbf{#1}}
\expandafter\def\csname cl#1\endcsname{\mathcal{#1}}
}
\expandafter\command@factory\next
}

\newcommand{\tmop}[1]{\ensuremath{\operatorname{#1}}}

\newcommand{\cdummy}{\cdot}

\usepackage[colorlinks]{hyperref}

\begin{document}

\title{Global well-posedness of the 3D Navier--Stokes equations perturbed by a deterministic vector field}

\author{Franco Flandoli\footnote{Email: franco.flandoli@sns.it. Scuola Normale Superiore of Pisa, Piazza dei Cavalieri 7, 56124 Pisa, Italy.}\quad
Martina Hofmanov\'{a}\footnote{Email: hofmanova@math.uni-bielefeld.de. Fakult\"{a}t f\"{u}r Mathematik, Universit\"{a}t Bielefeld, Postfach 100131, D-33501 Bielefeld, Germany. M.H. gratefully acknowledges the financial support by the German Science Foundation DFG via the Collaborative Research Center SFB 1283 and the Research Unit FOR 2402.}\quad
Dejun Luo\footnote{Email: luodj@amss.ac.cn. Key Laboratory of RCSDS, Academy of Mathematics and Systems Science, Chinese Academy of Sciences, Beijing, China and School of Mathematical Sciences, University of the Chinese Academy of Sciences, Beijing, China.}\quad
Torstein Nilssen\footnote{Email: torstein.nilssen@uia.no. Institute of Mathematics, University of Agder, Norway.}}

\maketitle

\begin{abstract}
We are concerned with the problem of global well-posedness of the 3D Navier--Stokes equations on the torus with unitary viscosity.
While a full answer to this question seems to be out of reach of the current techniques, we establish a regularization by a deterministic vector field. More precisely, we consider  the vorticity form of the system perturbed by  an additional transport type term. Such a perturbation conserves the enstrophy and therefore a priori it does not imply any smoothing. Our main result is a construction of a deterministic vector field $v=v(t,x)$ which  provides the desired regularization of the system and yields global well-posedness for large initial data outside arbitrary small sets. The proof relies on probabilistic arguments developed  by Flandoli and Luo, tools from rough path theory  by Hofmanov\'a, Leahy and Nilssen and a new Wong--Zakai approximation result, which itself combines probabilistic and rough path techniques.

\end{abstract}


\section{Introduction}

The problem of global well-posedness of the three dimensional Navier--Stokes system  describing  flows of incompressible fluids remains an outstanding open problem of great interest. Recently, it experienced a major breakthrough due to Buckmaster and Vicol \cite{BuVi} and Buckmaster, Colombo and Vicol \cite{BuCoVi}, who were able to prove nonuniqueness in a class of weak solutions. The~result followed a series of important works on the Euler equations by Buckmaster, De~Lellis, Isett, Sz\'ekelyhidi Jr. and Vicol \cite{BDSV, DelSze2,DelSze3,DelSze13, Ise} proving among others the Onsager conjecture.
The solutions to the Navier--Stokes system constructed by Buckmaster and Vicol \cite{BuCoVi} do not satisfy the corresponding energy inequality, i.e., they are not the so-called Leray solutions. Therefore, they  can be regarded as nonphysical. The problem of uniqueness of Leray solutions is one of the major challenges in the mathematical fluid dynamics research.

The systems of Navier--Stokes and Euler equations are derived from the basic physical principles. However, the derivation proceeds under a number of simplifying assumptions and in particular many physical parameters are neglected; this might be the potential reason for ill-posedness of weak solutions mentioned above. These considerations partly explain the motivation of introducing random perturbations to the equations and lead to the theory of regularization by noise, with the hope that noise may restore well-posedness of the systems. On the contrary, in the present paper we give an affirmative answer to the following question:

\begin{quotation}
\noindent\emph{Is there an additional enstrophy preserving deterministic term which provides global well-posedness for the 3D Navier--Stokes equations?}
\end{quotation}

Our guiding principle is the enstrophy conservation of the  added term together with the fact that the equations shall remain deterministic. More precisely, even though our result is in the spirit of \emph{regularization by noise} and our proof makes an essential use of probabilistic arguments, the~constructed perturbation is deterministic. Furthermore, the perturbation is driven by a  vector field which is a time derivative of a highly oscillatory but an explicit piecewise linear function. As a consequence, the result can be  further underlined by numerical simulations, which is one of the reasons why we aimed for the piecewise linear setting and  deterministic vector fields. However, by a slight modification, our proof yields regularization by smooth deterministic vector fields as well, see Remark~\ref{rem:approx}.

\bigskip

Consider the 3D Navier--Stokes equations on the torus $\bT^3$ in vorticity form%
\begin{equation}
\partial_{t}\xi+\mathcal{L}_{u}\xi=\Delta\xi\label{det NS1}%
\end{equation}
with an initial condition $\xi_{0}\in H$. Here $u$ and $\xi$ are the velocity and vorticity of the fluid, respectively, and $\mathcal{L}_{u}\xi= u\cdot \nabla \xi- \xi \cdot \nabla u$ is the Lie derivative. We write $H$ for the space of square integrable divergence free vector fields on $\bT^3$, see Section \ref{subsec-funct-spaces} for details.  Without loss of generality, we
focus on unitary viscosity. On the formal level, one can study the time evolution of the enstrophy and derive  the inequality
\begin{equation}
\frac{d}{d t} \Vert \xi (  t ) \Vert _{H}^{2}\leq
C_{1} \Vert \xi (  t ) \Vert _{H}^{6}.\label{energy est}%
\end{equation}
This is the key inequality  which provides a local bound yielding the maximal existence and uniqueness of a
solution of class $C\left(  [0,\tau);H\right)  $. The final time $\tau$ is not known to
be infinite or finite. We have denoted the constant by $C_{1}$ to remind
ourselves that this is a constant coming from unitary viscosity. The only
available lower bound on $\tau$ due to (\ref{energy est}) is a finite value
$\tau^{\ast}\left(  C_{1},\left\Vert \xi_{0}\right\Vert _{H}\right)  $
depending on $C_{1}$ and $\left\Vert \xi_{0}\right\Vert _{H}$.

Let  $v=v\left(  t,x\right)  $ be a given vector field, possibly random, periodic and divergence free; we consider the following modified model:%
\begin{equation}
\partial_{t}\xi+\mathcal{L}_{u}\xi=\Delta\xi+\Pi(v\cdot\nabla\xi
),\label{NS random}%
\end{equation}
where $\Pi$ is the Leray projection from $L^2(\bT^3,\bR^3)$ to $H$. The projection $\Pi$ is necessary since all the other terms are divergence free. We stress that the added perturbation is of transport type. Therefore, for a general vector field $v$, it does not have any smoothing effect, nor does it extend the lifespan of solutions. Indeed, due to the divergence-free constraint for $v$ it follows
\[
\left\langle \Pi(v\cdot\nabla\xi),\xi\right\rangle_H =0.
\]
Hence, the energy type estimate on $\Vert \xi(  t) \Vert^{2} _{H}$
for this model is the same as above, namely (\ref{energy est}) with the same
constant $C_{1}$. Hence, a priori, the only available lower bound on the
maximal time of well-posedness in $H$ is the same value $\tau^{\ast}\left(
C_{1},\left\Vert \xi_{0}\right\Vert _{H}\right)  $ as  above.

The main result of this paper significantly improves what simple energy type estimates can do. To state the next theorem, we take an arbitrary probability measure $\mu$ on the closed ball $B_H(K)= \{\xi_0\in H: \|\xi_0 \|_H\leq K\}$, and denote by $\tau_v(\xi_0)$ the life time of the solution to \eqref{NS random} starting from $\xi_0$.

\begin{theorem}\label{thm:main}
Let $K,T>0$ be given. For any $\ep>0$, there is a deterministic vector field $v$ such that
  $$\mu(\{\xi_0\in B_H(K): \tau_v(\xi_0)\geq T\}) \geq 1-\ep.$$
In other words, the collection of those initial data $\xi_{0}\in B_H(K)$ such that the system \eqref{NS random} is well-posed on $[0,T]$ has a $\mu$-measure greater then $1-\ep$.
\end{theorem}

Theorem \ref{thm:main} is the first result proving   a regularization by \emph{deterministic} vector fields for the 3D Navier--Stokes equations. Unfortunately, we are unable to show that there exists a deterministic vector field $v$ such that, for any $\xi_0 \in B_H(K)$, the equation \eqref{NS random} has a global solution on $[0,T]$, see Remark \ref{subsec-3-rem} for some discussions. By a different method the authors in \cite{Iyer} obtained a regularization by a deterministic transport term for several classes of equations. However, due to the Leray projection, the Navier--Stokes system is more delicate and the method does not apply.

Several results exist on regularization by
noise where the noise is white in time and multiplication operations are of
Stratonovich type. Such results are especially appreciated when Stratonovich models are
accepted as an idealization of real models, that is,  a fast varying term is replaced by white noise in time. However, for practical purposes as well as numerical simulations  this is not satisfactory since one is forced to go back and replace white noise by a suitable smooth approximation. In other words, certain stability with respect to the driving noise is necessary which translates to the so-called Wong--Zakai principle.
Our proof of  Theorem \ref{thm:main} is motivated by these considerations.
More precisely, two ingredients are needed: a stochastic regularization by noise and a Wong--Zakai principle.

On the one hand, in a previous work of two members of our team, namely Flandoli and Luo \cite{FL19}, regularization by a transport  noise of Stratonovich type was established for the 3D Navier--Stokes system. Remarkably, by a very delicate argument it was possible to show that a suitable noise increases the dissipation of the system with a large probability. Practically, this  translates to an increased viscosity which in turn extends the lifespan of the solution. We refer to Section~\ref{ss:formul} for a more detailed discussion of the results of \cite{FL19}.

On the other hand, the Wong--Zakai principle for stochastic partial differential equations became significantly more accessible by the recent advances in the  theory of rough paths. In the context of the Navier--Stokes system, the theory was developed by the other two members of our author team, namely Hofmanov\'a and Nilssen (together with Leahy) \cite{HLN,HLN19}. In these works, certain rough perturbations of transport type were included in the model and existence of solutions was proved, proceeding by a mollification of the noise.
Additional results including the Wong--Zakai principle were obtained in the two dimensional setting.

It turns out that the noise which provides regularization in \cite{FL19} was not treated in \cite{HLN,HLN19}. More importantly, it is not clear how to obtain the necessary Wong--Zakai principle in the three dimensional setting
directly by the techniques of \cite{HLN,HLN19}. Indeed,  to this end it would be necessary to establish at least local-in-time uniqueness of strong solutions to the rough path formulation of the equations. Otherwise one could only formulate a certain Wong--Zakai principle up to a subsequence which depends on the randomness variable $\omega$. Thus, measurability with respect to $\omega$ may be lost.

To this time we are not able to prove the necessary uniqueness  in the rough path setting. Thus, we proceed differently. The  idea is to make use of the corresponding uniqueness result in the stochastic setting, which is surprisingly easy to establish. To this end, a further combination of rough path techniques with the stochastic compactness method based on the Skorokhod representation theorem is necessary. Especially, the identification of the limiting equation in our main Proposition~\ref{prop-intermediate} below manifests the nice interplay between probabilistic arguments based on the martingale theory and pathwise arguments relying on the  theory of rough paths.

As an intermediate result towards the proof of Theorem~\ref{thm:main} we obtain the following statement which reflects the probabilistic nature of our construction and which is interesting in its own right.

\begin{theorem}\label{thm:main1}
Given $K,T,\varepsilon>0$,
there is a random vector field $v= v(\omega, t ,x)$ (see the formula below Corollary \ref{coro-1}), such that for every $\xi_{0}\in B_H(K)$, the maximal time of well-posedness in $H$ for equation \eqref{NS random}  with initial condition
$\xi_{0}$ is greater than $T$, with probability $1-\varepsilon$.
\end{theorem}

We point out  that the time regularity of the constructed vector field $v$ in Theorem~\ref{thm:main} is such that \eqref{NS random} is understood and solved in the classical deterministic sense. In other words, it is \emph{not} a stochastic partial differential equation. It will be seen in the proof below that $v$ involves the time derivative of a suitable regularization of some Brownian motions defined in Section \ref{ss:BM}.

The paper is organized as follows. In Section 2 we make some preparations concerning the functional analytical setting, the explicit choice of a complete orthonormal system of $H$, and the elements of rough path theory. In Section 3 we first recall the main results of \cite{FL19}, and then state in a more precise way the equations and a series of intermediate results needed for proving Theorem \ref{thm:main}. Section 4 is devoted to the proof of the Wong--Zakai approximation: Theorem \ref{wong zakai in law}, which is the main ingredient in the proof of Theorem \ref{thm:main}.

\section{Preliminaries}

\subsection{Function spaces}\label{subsec-funct-spaces}

For a given $m\in \bR$ and $d,D\in \bN$, we denote $W^{m,2}(\bT^d; \bR^{D})=(I-\Delta)^{-\frac{m}{2}}L^2( \bT^d ; \bR^{D})$.
We denote by $H^{m}$ the subspace of $W^{m,2}(\bT^{d};\bR^{d})$  consisting of divergence free vector fields, i.e.,
$$
H^m=\left\{f\in W^{m,2}(\bT^{d};\bR^{d}); \;\nabla \cdot  f= 0\right\},
$$
and let $\|\cdot\|_{H^{m}}$ be the corresponding norm. We write $H$ for $H^{0}$.
In order to analyze the convective  term in the  Navier--Stokes system, we employ the classical notation and bounds.
Owing to Lemma~2.1 in \cite{RT83}, the trilinear form
$$
b(u,\varv,w)= \int_{\bT^d} ((u\cdot \nabla)\varv)\cdot w \, \, d x=\sum_{i,j=1}^d\int_{\bT^d} u^i \partial_{x_{i}}\varv^j w^j \,\, d x
$$
satisfies the continuity property
\begin{equation}\label{trilinear form estimate}
|b(u,v,w)|\lesssim_{m_1,m_2,m_3,d}
\|u\|_{H^{m_1}}\|v\|_{H^{m_2+1}}\|w\|_{H^{m_3}}, \qquad m_1 + m_2 + m_3 > \frac{d}{2} , \quad m_1,m_2,m_3 \geq 0 .
\end{equation}
Moreover, for all $u\in  H^{m_1}$ and  $(\varv,w)\in W^{m_2+1,2}\times W^{m_3,2}$ such that $m_1,m_2,m_3$ satisfy \eqref{trilinear form estimate}, we have
\begin{equation}\label{eq:B prop}
b(u,\varv,w)=-b(u,w,\varv) \quad \textnormal{and} \quad b(u,\varv,\varv)=0.
\end{equation}

\subsection{A basis of $H$ and complex Brownian motions}
\label{ss:BM}

Recall that $\mathbb Z^3_0= {\mathbb Z}^3\setminus \{0\}$ is the nonzero lattice points. Let ${\mathbb Z}^3_0= {\mathbb Z}^3_+ \cup {\mathbb Z}^3_-$ be a partition of ${\mathbb Z}^3_0$ such that
  $${\mathbb Z}^3_+ \cap {\mathbb Z}^3_-= \emptyset, \quad {\mathbb Z}^3_+ = - {\mathbb Z}^3_-.$$
Let $L^2_0(\bT^3, \mathbf C)$ be the space of complex valued square integrable functions on $\bT^3$ with zero average. It has the complete orthonormal system:
  $$e_k(x)= {\rm e}^{2\pi {\rm i} k\cdot x}, \quad x\in \bT^3,\, k\in {\mathbb Z}^3_0,$$
where ${\rm i}$ is the imaginary unit. For any $k\in {\mathbb Z}^3_+$, let $\{a_{k,1}, a_{k,2}\}$ be an orthonormal basis of $k^\perp := \{x\in \bR^3: k\cdot x=0\}$ such that $\big\{a_{k,1}, a_{k,2}, \frac{k}{|k|} \big\}$ is right-handed. The choice of $\{a_{k,1}, a_{k,2}\}$ is not unique. For $k\in {\mathbb Z}^3_-$, we define $a_{k,\alpha} = a_{-k,\alpha}$, $\alpha=1,2$. Now we can define the divergence free vector fields:
  \begin{equation}\label{vector-fields}
  \sigma_{k,\alpha}(x) = a_{k,\alpha} e_k(x), \quad x\in \bT^3,\,  k\in {\mathbb Z}^3_0,\, \alpha=1,2.
  \end{equation}
Then $\big\{\sigma_{k,1}, \sigma_{k,2}:  k\in {\mathbb Z}^3_0 \big\}$ is a CONS of the subspace $H_{\mathbf C} \subset L^2_0(\bT^3,\mathbf C^3)$ of square integrable and divergence free vector fields with zero mean. A vector field
  $$v= \sum_{k\in \mathbb Z^3_0} \sum_{\alpha=1}^2 v_{k,\alpha} \sigma_{k,\alpha} \in H_{\mathbf C}$$
has real components if and only if $\overline{v_{k,\alpha}} = v_{-k, \alpha}$.

Next we introduce the family $\big\{W^{k,\alpha}: k\in {\mathbb Z}^3_0, \alpha=1,2 \big\}$ of complex Brownian motions. Let
  $$\big\{B^{k,\alpha}: k\in {\mathbb Z}^3_0,\, \alpha=1,2 \big\}$$
be a family of independent standard real Brownian motions; then the complex Brownian motions can be defined as
  $$W^{k,\alpha} = \begin{cases}
  B^{k,\alpha} + {\rm i} B^{-k,\alpha}, & k\in  {\mathbb Z}^3_+;\\
  B^{-k,\alpha} - {\rm i} B^{k,\alpha}, & k\in  {\mathbb Z}^3_-.
  \end{cases}$$
Note that $\overline{W^{k,\alpha}}= W^{-k,\alpha}\, (k\in{\mathbb Z}^3_0, \alpha=1,2)$, and they have the following quadratic covariation:
  \begin{equation}\label{qudratic-var}
  \llangle W^{k,\alpha}, W^{l,\beta} \rrangle_t= 2\, t\, \delta_{k,-l}\, \delta_{\alpha, \beta},\quad k,l\in {\mathbb Z}^3_0,\, \alpha, \beta\in \{1,2\} .
  \end{equation}

\subsection{Elements of rough paths theory} \label{ss:rp}

Let $I\subset \bR$ be a bounded interval  and let $E$ be a Banach space with a norm $\| \cdot \|_E$. For a path $g:I\to E$ we define its increment as $\delta g_{s t}:=g_{t}-g_{s}$, $s,t\in I$. Let
$\Delta_I := \{(s,t)\in I^2: s\le t\}$. For a two-index map $g: \Delta_{I}\rightarrow E$, we define the second order increment operator $\delta g_{s \theta t} = g_{st} - g_{\theta t} - g_{s \theta}$, $s \leq \theta \leq t.$

Let $\alpha>0$. We denote by $C_{2}^{\alpha}(I; E)$ the closure of the set of smooth 2-index maps $g: \Delta_I \rightarrow E$ with respect to the seminorm
$$
[g]_{\alpha}:=[g]_{\alpha,I,E}:=\sup_{s,t\in \Delta_{I},s\neq t}\frac{\|g _{s t}\|_E}{|t-s|^{\alpha}}<\infty.
$$
By $C^{\alpha}_{2,\rm{loc}}(I;E)$ we denote the space of 2-index maps $g:\Delta_{I}\to E$ such that there exists a  covering $\{I_{k}\}_{k}$ of the interval $I$ so that $g\in C^{\alpha}_{2}(I_{k};E)$ for every $k$.
By $C^{\alpha}(I; E)$ we denote the closure of the set of smooth paths $g : I \rightarrow E$ with respect to the seminorm $[\delta g]_{\alpha}$.
Note that with this definition, the spaces $C^{\alpha}(I; \bR^{m})$ and $C_{2}^{\alpha}(I; \bR^{m})$, $m\in\bN$, are Polish.

Next, we present the definition of a rough path. A detailed exposition of rough path theory can be found in \cite{FrHa14}.

\begin{definition}\label{defi-rough-path}
Let $T>0$, $m \in\bN $ and $\alpha\in(\tfrac13,\frac12]$. An $\alpha$-rough path is a  pair
\begin{equation}\label{p-var-rp}
\bZ=(Z, \mathbb{Z}) \in C_{2}^{\alpha} ([0,T];\bR^{m}) \times C^{2\alpha}_2 ([0,T]; \bR^{m\times m})
\end{equation}
satisfying the Chen's relation
\begin{equation*}\label{chen-rela}
\delta \mathbb{Z}_{s\theta t}=Z_{s\theta} \otimes Z_{\theta t},   \qquad  s \leq \theta \leq t.
\end{equation*}
Given a smooth path $z$, there is a canonical lift to a rough path $(Z, \mathbb{Z})$ given by
$$Z_{st}:= \delta z_{st} \quad \textnormal{ and } \quad \mathbb{Z}_{st}:=\int_s^t \delta z_{s r} \otimes \dot{z}_{r}\, d r ,$$
for which Chen's relation is readily checked.
An $\alpha$-rough path $\mathbf{Z}=(Z, \mathbb{Z})$ is said to be geometric if it can be obtained as the limit in the  product topology $C_{2}^{\alpha} ([0,T];\bR^{m}) \times C^{2\alpha}_2 ([0,T]; \bR^{m\times m})$  of a sequence of rough paths  $\{(Z^{n},\mathbb{Z}^{n})\}_{n=1}^{\infty}$
which are canonical lifts of some smooth paths $z^n:[0,T] \to \bR^m$.
\end{definition}

We proceed with a definition of an unbounded rough driver, which can be regarded as an operator valued rough path taking values in unbounded operators. In view of the application to the Navier--Stokes system we work directly with the scale of Hilbert spaces $H^{n}$ defined in Section~\ref{subsec-funct-spaces}.

\begin{definition}
\label{def:urd}
Let $T>0$ and $\alpha\in (\frac13,\frac12]$. An unbounded $\alpha$-rough driver is a pair $\mathbf{A} = (A^1,A^2)$ of $2$-index maps satisfying:
there exists a constant $C_{A}>0$ such that for every $0\le s \le t \le T$
\begin{equation}\label{ineq:UBRcontrolestimates}
\begin{aligned}
| A^1_{st}|_{\mathcal{L}(H^{-n},H^{-(n+ 1)})}& \leq C_{{A}}|t-s|^{\alpha} \ \  \text{for}\ \ n\in\{0,2\} ,\\
|A^2_{st}|_{\mathcal{L}(H^{-n},H^{-(n+2)})}& \leq C_{{A}}|t-s|^{2\alpha} \ \ \text{for}\ \ n\in\{0,1\},
\end{aligned}
\end{equation}
and   Chen's relation holds true, namely,
\begin{equation}\label{eq:chen-relation}
\delta A^1_{s\theta t}=0,\qquad\delta A^2_{s\theta t}= A^1_{\theta t}A^1_{s\theta},\qquad 0\leq s\le\theta\le t\leq T.
\end{equation}
\end{definition}

The partial differential equations of interest in this paper can be written in the abstract form
\begin{equation} \label{eq:AbstractURDEquation}
d g_t =  d\mu_t + \mathbf{A}(d t) g_t,
\end{equation}
where $\mu$ denotes the corresponding drift of appropriate spatial regularity and depending also on the solution $g$ and $\mathbf{A}$ is an unbounded $\alpha$-rough driver (cf. \eqref{expansion equation} below).
We say that a  path
$g: [0,T] \rightarrow H$ is a solution to \eqref{eq:AbstractURDEquation} provided the 2-index map
$$
g_{st}^{\natural} := \delta g_{st} - \delta \mu_{st} - A_{st}^1 g_s - A_{st}^2 g_s,\qquad 0\leq s\leq t\leq T,
$$
belongs to $C_{2,\rm{loc}}^{3\alpha}([0,T];H^{-3})$.

We conclude this section with the main a priori estimate which can be proved following the lines of \cite[Corollary 2.11]{DeGuHoTi16}, cf. \cite{HLN,HLN19}. The bound which holds global in time is a consequence of the local estimate combined with \cite[Exercise 4.24]{FrHa14}.

\begin{theorem} \label{Thm2.5Abstract}
Assume
\begin{itemize}

\item
$\mathbf{A}  = (A^1,A^2)$ is an unbounded $\alpha$-rough driver for some $\alpha \in (\frac13,\frac12]$;

\item
$\mu: [0,T] \rightarrow H^{-2}$ is Lipschitz continuous and $\mu: [0,T] \rightarrow H^{-1}$ is $\frac12$-H\"older continuous, i.e.,
$$\| \delta \mu_{st} \|_{H^{-2}} \leq C_{\mu,1}|t-s|\qquad \text{and}\qquad \| \delta \mu_{st} \|_{H^{-1}} \leq C_{\mu,2}|t-s|^{\frac{1}{2}}$$ for some constants $C_{\mu,1}$ and $C_{\mu,2}$;
\item
a bounded path $g: [0,T] \rightarrow H$ is a solution to \eqref{eq:AbstractURDEquation}.
\end{itemize}
Then there exists a constant $L>0$ such that whenever $C_{A}|t-s|^{\alpha} \leq L$  we have
$$
\|g_{st}^{\natural}\|_{H^{-3}} \lesssim \big(|g|_{L^{\infty}_T H}(1+C_{A}^{2})  + C_{\mu,1}C_{A}\big)|t-s|^{3\alpha},
$$
$$
\|\delta g_{s t}\|_{H^{-1}}\lesssim \big(|g|_{L^{\infty}_T H}(1+C_{A}^{2})+C_{\mu,2}+C_{\mu,1}C_{A}\big)|t-s|^{\alpha},
$$
where the implicit constants as well as $L$ are universal and in particular independent of $g$ and $\mathbf{A}$.

Finally, we have the following bound which holds globally in time
$$
\| \delta g_{st} \|_{H^{-1}} \lesssim \left(|g|_{L^{\infty}_T H}(1+C_{A}^{2})+C_{\mu,2}+C_{\mu,1}C_{A}\right) \left( 1 + C_A^{\frac{1- \alpha}{\alpha}} \right)|t-s|^{\alpha}.
$$
\end{theorem}

\section{Formulation of the main results}
\label{ss:formul}

Let $\ell^2= \ell^2(\mathbb Z^3_0)$ be the space of square summable sequences indexed by $\mathbb Z^3_0$; it is equipped with the norm $\|\cdot \|_{\ell^2}$. For any $N\in \mathbb Z_+$, take $\theta^N\in \ell^2$ such that
  $$\theta^N_k= \frac{{\mathbf 1}_{\{ N\leq |k|\leq 2N\}}}{|k|^\gamma}, \quad k\in \mathbb Z^3_0,$$
where $\gamma>0$ is some fixed constant. On a given time interval $[0, T]$, we consider the stochastic 3D Navier--Stokes equations with a transport type noise:
\begin{equation}  \label{SPDE no cut-off}
  d \xi^N +\mathcal{L}_{u^N} \xi^N \hspace{0.17em} d t =
  \Delta \xi^N \hspace{0.17em} d t + \frac{C_{\nu}}{\| \theta^N
  \|_{\ell^2}}  \sum_{k \in \mathbb{Z}_0^3} \sum_{\alpha = 1}^2 \theta_k^N \Pi
  (\sigma_{k, \alpha} \cdot \nabla \xi^N) \circ d W_t^{k, \alpha},
\end{equation}
with $\xi^N(0)= \xi_0\in H$. Here $C_{\nu} = \sqrt{3 \nu / 2}$  with some given $\nu > 0$, which will be chosen depending on the size of the initial conditions.

As already discussed in the introduction, due to the presence of the nonlinear term, these equations have only local solutions in $H$. Thus we make use of a cut-off technique. For $R>0$, let $f_R\in C^1_b([0,\infty), [0,1])$ be a non-increasing function such that it is identically 1 on $[0,R]$ and vanishes on $[R+1,\infty)$. Consider the equations with cut-off:
\begin{equation} \label{SPDE with cut-off}
  d \xi_R^N + f_R (\| \xi_R^N \|_{- \delta}) \mathcal{L}_{u_R^N}
  \xi_R^N \hspace{0.17em} d t = \Delta \xi_R^N \hspace{0.17em}
  d t + \frac{C_{\nu}}{\| \theta^N \|_{\ell^2}}  \sum_{k \in
  \mathbb{Z}_0^3} \sum_{\alpha = 1}^2 \theta_k^N \Pi (\sigma_{k, \alpha} \cdot
  \nabla \xi_R^N) \circ d W_t^{k, \alpha},
\end{equation}
where $\|\cdot \|_{-\delta} = \|\cdot \|_{H^{-\delta}}$. Recall that $B_H(K)$ is the closed ball in $H$ centered at the origin with radius $K>0$. It was shown in Theorem 1.3 of \cite{FL19} that, for every $N\geq 1$ and $\xi_R^N(0)= \xi_0\in B_H(K)$, the above equation has a pathwise unique solution satisfying
  \begin{equation} \label{apriori-estimate}
  \PP\mbox{-a.s.}, \quad \sup_{t \in [0, T]} \| \xi_R^N (t) \|_H^2 + \int_0^T \| \nabla \xi_R^N
     (t) \|_H^2 \,d t \leq C (R, K),
  \end{equation}
where $C (R, K)>0$ is some deterministic constant.

Moreover, given $K>0$, we deduce from \cite[Theorem~1.4]{FL19} that for all $R>0$ and $\nu>0$ big enough, for any $\varepsilon>0$, it holds
  \begin{equation} \label{convergence-in-probab}
  \lim_{N\to \infty} \sup_{\xi_0\in B_H(K)} \mathbb P \bigg(\sup_{t\in [0,T]} \big\|\xi_R^N (t,\xi_0)- \xi(t,\xi_0) \big\|_{-\delta} >\varepsilon \bigg) =0,
  \end{equation}
where $\xi(t,\xi_0)$ is the unique solution to the following deterministic 3D Navier--Stokes equation
\begin{equation} \label{det NS}
  \partial_t \xi +\mathcal{L}_u \xi = \bigg(1 + \frac35 \nu\bigg) \Delta \xi,\quad \xi(0)=\xi_0.
\end{equation}
It is well known that, for given $K>0$, there exists $\nu>0$ such that for all $\xi_0\in B_H(K)$, the equation \eqref{det NS} admits a unique solution satisfying
  \begin{equation} \label{det-NS-growth}
  \sup_{t \in [0, T]} \| \xi (t) \|_H^2 + \int_0^T \| \nabla \xi (t) \|_H^2 \,
     d t \leq C (K)^2 .
  \end{equation}
In the sequel, $K$ and a corresponding $\nu$ will be considered as fixed.

Choose $R_K= C (K)+2$; we deduce from the assertions \eqref{convergence-in-probab} and \eqref{det-NS-growth} that, given $\varepsilon>0$, there is $N_0= N_0(K,\varepsilon)\in \mathbb Z_+$ such that for all $N\geq N_0$, for all $\xi_0\in B_H(K)$,
  \begin{equation}\label{bound-vorticity}
  \mathbb P \bigg(\sup_{t\in [0,T]} \big\|\xi_{R_K}^N (t,\xi_0) \big\|_{-\delta} \leq R_K- 1 \bigg) \geq 1-\varepsilon,
  \end{equation}
where we have used the fact that $\|y\|_{-\delta} \leq \|y\|_H$. This implies that, for every $N\geq N_0$, $\xi_{R_K}^N$ solves the equation \eqref{SPDE no cut-off} without cut-off with a probability greater than $1-\varepsilon$.

Next, for $k\in \mathbb Z^3_0$ and $\alpha=1,2$, let $\big\{ W_t^{k, \alpha, n} \big\}_{n\geq 1}$ be a piecewise linear approximation of the Brownian motion $W_t^{k, \alpha}$.

\begin{remark}\label{rem:approx}
Alternatively, we may replace piecewise linear approximations by mollifications, in which case the obtained vector fields $v$ in Theorem~\ref{thm:main} and Theorem~\ref{thm:main1} are smooth.
\end{remark}

 We consider the 3D Navier--Stokes equations with smooth random force
\begin{equation}\label{random PDE without cut-off}
  \partial_t \xi^{N, n} +\mathcal{L}_{u^{N, n}} \xi^{N, n} = \Delta \xi^{N, n}
  + \frac{C_{\nu}}{\| \theta^N \|_{\ell^2}}  \sum_{k \in \mathbb{Z}_0^3}
  \sum_{\alpha = 1}^2 \theta_k^N \Pi (\sigma_{k, \alpha} \cdot \nabla \xi^{N,
  n}) \partial_t W_t^{k, \alpha, n} ,
\end{equation}
as well as the equations with cut-off
\begin{equation}\label{random PDE with cut-off}
  \partial_t \xi_R^{N, n} + f_R (\| \xi_R^{N, n} \|_{- \delta})
  \mathcal{L}_{u_R^{N, n}} \xi_R^{N, n} = \Delta \xi_R^{N, n} +
  \frac{C_{\nu}}{\| \theta^N \|_{\ell^2}}  \sum_{k \in \mathbb{Z}_0^3}
  \sum_{\alpha = 1}^2 \theta_k^N \Pi (\sigma_{k, \alpha} \cdot \nabla
  \xi_R^{N, n}) \partial_t W_t^{k, \alpha, n}.
\end{equation}
Similarly to \eqref{SPDE with cut-off}, for any $\xi_R^{N, n}(0)= \xi_0\in B_H(K)$, the latter equation admits a unique solution verifying
  \begin{equation}\label{random PDE with cut-off energy bound}
  \PP\mbox{-a.s.}, \quad  \sup_{t \in [0, T]} \| \xi_R^{N, n} (t) \|_H^2 + \int_0^T \| \nabla \xi_R^{N, n} (t) \|_H^2\, d t
    \leq C (R, K).
  \end{equation}

Note that the estimate \eqref{random PDE with cut-off energy bound} depends only on $R$ and the bound $K$ of the initial condition $\xi_0\in H$ and is independent of $N,n$. The basis for our Wong--Zakai result is obtained by techniques from rough path theory developed in \cite{HLN, HLN19}, which allow us to derive additional  estimates uniform in $n$ (see Section~\ref{ss:rp} and Proposition \ref{prop:apriori}). However, rough path theory alone  is not sufficient to conclude. In particular, the obtained bounds only permit to deduce   \emph{relative} compactness of realizations of the approximate sequence of solutions $\{\xi^{N_{0},n}_{R_{K}}(\omega)\}_{n\geq1}$ and the  convergence follows only for a  subsequence which depends on $\omega$.

In order to obtain convergence of the full sequence, one would need to establish uniqueness of the rough path formulation of the limiting (as $n\to\infty$) equation \eqref{SPDE with cut-off}. This is a very challenging problem which remains open. The main difficulty lies in the presence of the Leray projection which is not compatible with the tensorization technique developed in \cite{DeGuHoTi16} to prove uniqueness of variational  rough PDEs.

To overcome this issue, we reach back to probability theory and proceed by a stochastic compactness argument relying on Skorokhod representation theorem. The point is that uniqueness for the stochastic formulation of \eqref{SPDE with cut-off} follows by classical arguments. Nevertheless, it is necessary to preserve the rough path formulation of the equations in the core of the proof, as this is the setting where we are able to rigorously obtain the convergence of \eqref{random PDE with cut-off} to \eqref{SPDE with cut-off}. The final step entails a new identification of the limit procedure which combines martingale and rough path arguments.

More precisely, first we prove  the following result (see \emph{Step 1} of the proof of Proposition~\ref{prop-intermediate}).

\begin{lemma} \label{lem-1}
Assume that the sequence $\{\xi^n_0 \}_{n\geq 1} \subset B_H(K)$. Let $\xi_{R_K}^{N_0, n}$ be the unique solution to \eqref{random PDE with cut-off} with $N=N_0$, $R=R_K$ and $\xi_{R_K}^{N_0, n}(0)= \xi^n_0$. Then the family of laws of $\xi_{R_K}^{N_0, n}$ is tight in $C ([0, T] ; H^{- \delta}) \cap L^2(0,T; H)$.
\end{lemma}

Consequently, our main Wong--Zakai approximation result  proved in Section \ref{sec-Wong-Zakai} reads as follows.

\begin{theorem}  \label{wong zakai in law}
Let $\xi_{R_K}^{N_0, n}(t,\xi_0)$ (resp. $\xi_{R_K}^{N_0}(t,\xi_0)$) be the unique solution to \eqref{random PDE with cut-off} (resp. \eqref{SPDE with cut-off}) with the initial value $\xi_0\in H$. Then for any $\varepsilon>0$,
  $$\lim_{n\to \infty} \sup_{\xi_0\in B_H(K)} \mathbb{P} \bigg(\sup_{t\in[0,T]}\big\|\xi_{R_K}^{N_0, n}(t,\xi_0) - \xi_{R_K}^{N_0} (t,\xi_0)\big\|_{-\delta} >\varepsilon \bigg)=0. $$
\end{theorem}

Combining this assertion with \eqref{bound-vorticity}, we obtain the following.

\begin{corollary}
  Given $\varepsilon > 0$ there is $n_0 = n_0 (K, \varepsilon)$ such that for every $n \geq n_0$, for all $\xi_0\in B_H(K)$,
  \[ \mathbb{P} \bigg( \sup_{t\in[0,T]}\big\| \xi_{R_K}^{N_0, n} (t,\xi_0) \big\|_{-\delta}
     \leq R_K \bigg) \geq 1 - 2 \varepsilon . \]
  In particular, $\xi_{R_K}^{N_0, n_0}$ is a global (namely on $[0, T]$) solution of
  equation \eqref{random PDE without cut-off} with large probability.
\end{corollary}

Remark that $\xi_{R_K}^{N_0, n_0}$ satisfies also the bound \eqref{random PDE with cut-off energy bound}. Since the involved rough paths are smooth, one can follow the classical arguments to prove the uniqueness of solutions with such bounds, see for instance \cite[Theorem 1.3]{FL19}.

\begin{lemma}
  Consider equation \eqref{random PDE without cut-off} for some fixed values of the
  parameters $N,n$. Assume it has a weak solution $\xi^{N,n}$ on
  some interval $[0, T]$ so that, $ \mathbb{P}$-a.s.,
    $$\sup_{t \in [0, T]} \big\| \xi^{N, n} (t) \big\|_H^2 + \int_0^T \big\|\nabla \xi^{N, n} (t)\big\|_H^2 \,d t \leq C< \infty.$$
  Then this solution is pathwise unique.
\end{lemma}

Next, we derive the following consequence.

\begin{corollary}\label{coro-1}
  Given $K>0$ and $\varepsilon > 0$, there are $N_0$ and $n_0$ such that for all $\xi_0 \in B_H(K)$, the maximal time $\tau_{N_0, n_0}(\xi_0)$ of existence and uniqueness  for equation \eqref{random PDE without cut-off} in $H$  satisfies
  \[ \mathbb{P} (\tau_{N_0, n_0}(\xi_0) \geq T) \geq 1 - \varepsilon . \]
\end{corollary}

This permits to conclude the proof of Theorem~\ref{thm:main1} since it now suffices to define
\[
v(t,x):=\frac{C_{\nu}}{\| \theta^{N_{0}} \|_{\ell^2}}  \sum_{k \in \mathbb{Z}_0^3}
  \sum_{\alpha = 1}^2 \theta_k^{N_{0}}  \sigma_{k, \alpha}(x) \partial_t W_t^{k, \alpha, n_{0}}.
\]
Finally, we turn to the proof of Theorem~\ref{thm:main}. In what follows we formulate and prove two propositions, both of which lead to Theorem~\ref{thm:main}. The difference between Proposition~\ref{prop:thm1b} and Proposition~\ref{prop:thm1a} lies in the order of quantifiers. In the simpler Proposition~\ref{prop:thm1b}, the parameters $N_{0},n_{0}$ do not depend on $\mu$ but the price to pay is that the ``good'' $\omega$ is not easy to ``catch'', i.e. the proof is not constructive. On the other hand, Proposition~\ref{prop:thm1a} shows that  every $\omega$ from a set of full probability is ``good'' provided  the parameters $N_{0},n_{0}$ are appropriately chosen.

\begin{proposition}\label{prop:thm1b}
Let $K,T>0$ be given. For any $\varepsilon>0$  there exist  $N_{0},n_{0}\in\N$  and $\omega\in \Omega$  such that for every probability measure $\mu$ on $B_{H}(K)$ it holds
$$
\mu\big(\{\xi_{0}\in B_{H}(K)\,:\, \tau_{N_{0},n_{0}}(\xi_{0},\omega)\geq T \}  \big)\geq 1-\varepsilon.
$$
\end{proposition}

\begin{proof}
According to Corollary~\ref{coro-1}  there exist sufficiently large $N_{0}, n_{0}\in \N$ such that
  \begin{equation}\label{probab}
  \inf_{\xi_0 \in B_H(K)} \P (\tau_{N_{0}, n_{0}}(\xi_0) \geq T) \geq 1-\eps.
  \end{equation}
Hence
  $$\int_{B_H(K)} \PP (\tau_{N_{0}, n_{0}}(\xi_0) \geq T) \,\mathrm{d}\mu(\xi_0) \geq \inf_{\xi_0 \in B_H(K)} \PP (\tau_{N_{0}, n_{0}}(\xi_0) \geq T) \geq 1-\varepsilon$$
and Fubini's theorem implies
  $$\int_{\Omega} \bigg[\int_{B_H(K)} \textbf{1}_{\{\tau_{N_{0}, n_{0}}(\xi_0,\omega) \geq T\}} \,\mathrm{d}\mu(\xi_0) \bigg] \, \mathrm{d}\PP(\omega) \geq 1-\varepsilon.$$
Therefore, there exists $\omega\in \Omega$ such that
  $$\mu\big(\{ \xi_0 \in B_H(K): \tau_{N_{0}, n_{0}}(\xi_0,\omega) \geq T \} \big) \geq 1-\varepsilon. $$
\end{proof}

Alternatively, we may prove  Theorem~\ref{thm:main} as follows.

\begin{proposition}\label{prop:thm1a}
Let $K,T>0$ be given and let $\mu$ be a probability measure on $B_{H}(K)$. For any $\varepsilon>0$ and  a.e. $\omega\in\Omega$ there exist parameters $N_{0},n_{0}\in\N$ such that
$$
\mu\big(\{\xi_{0}\in B_{H}(K)\,:\, \tau_{N_{0},n_{0}}(\xi_{0},\omega)\geq T \}  \big)\geq 1-\varepsilon.
$$
\end{proposition}

\begin{proof}
For any $i\in \mathbb N$, by Corollary \ref{coro-1}, we can find $N_i, n_i\in \mathbb N$ big enough such that
  $$\sup_{\xi_0 \in B_H(K)} \PP (\tau_{N_i, n_i}(\xi_0) < T) \leq \frac1{2^i}. $$
Define the events
  $$A_i= \big\{(\xi_0,\omega)\in B_H(K)\times \Omega: \tau_{N_i, n_i}(\xi_0,\omega) <T \big\}, \quad i\in \mathbb N;$$
then for any $i\geq 1$, by Fubini's theorem,
  $$(\mu\otimes \PP) (A_i) = \int_{B_H(K)\times \Omega} \textbf{1}_{A_i}(\xi_0,\omega) \,d(\mu\otimes \PP) = \int_{B_H(K)} \PP (\tau_{N_i, n_i}(\xi_0) < T) \, d\mu(\xi_0) \leq \frac1{2^i}.$$
Thus $\sum_{i\geq 1} (\mu\otimes \PP) (A_i) \leq 1 <\infty$. By Borel-Cantelli's lemma, for $(\mu\otimes \PP)$-a.e. $(\xi_0,\omega)\in B_H(K)\times \Omega$, there is $i_0= i_0(\xi_0,\omega)$ such that
  $$(\xi_0,\omega) \in \bigcap_{i\geq i_0} A_i^c= \bigcap_{i\geq i_0} \{ \tau_{N_i, n_i}(\xi_0,\omega) \geq T \}. $$
Equivalently,
  $$(\mu\otimes \PP) \bigg(\bigcup_{N=1}^\infty \bigcap_{i= N}^\infty \{ \tau_{N_i, n_i}(\xi_0,\omega) \geq T \} \bigg) =1. $$
By Fubini's theorem, there exists a full probability event $\Omega_0\subset \Omega$ such that for all $\omega\in \Omega_0$,
  $$1= \mu\bigg(\bigcup_{N=1}^\infty \bigcap_{i= N}^\infty A_i^c(\omega) \bigg) =\lim_{N\to \infty} \mu\bigg(\bigcap_{i= N}^\infty A_i^c(\omega) \bigg), $$
where $A_i^c(\omega)= \{\xi_0 \in B_H(K): \tau_{N_i, n_i}(\xi_0,\omega) \geq T \}$ is the section of $A_i^c$ at $\omega$. In particular,
  \begin{equation}\label{eq-1}
  \lim_{i\to \infty} \mu\big( \{\xi_0 \in B_H(K): \tau_{N_i, n_i}(\xi_0,\omega) \geq T \} \big) =1.
  \end{equation}
      Now for any $\ep>0$, there exists $i$ big enough such that
  $$\mu\big( \{\xi_0 \in B_H(K): \tau_{N_i, n_i}(\xi_0,\omega) \geq T \} \big)\geq 1-\ep. $$
\end{proof}

%

  \begin{proof}[Proof of Theorem \ref{thm:main}]
  Finally, having found  $N_{0},n_{0}\in\N$ and $\omega\in\Omega$ as in Proposition~\ref{prop:thm1a} or Proposition~\ref{prop:thm1b}
 we obtain a \emph{deterministic} vector field
  $$v(\omega, t, x)= \frac{C_{\nu}}{\| \theta^{N_{i}} \|_{\ell^2}}  \sum_{k \in \mathbb{Z}_0^3}
  \sum_{\alpha = 1}^2 \theta_k^{N_{i}}  \sigma_{k, \alpha}(x) \partial_t W_t^{k, \alpha, n_{i}}(\omega) $$
and finish the proof of Theorem \ref{thm:main}.
  \end{proof}

 We complete this section with the following remark.

\begin{remark}\label{subsec-3-rem}
It is natural to ask whether there exists a deterministic vector field $v$ such that, for any $\xi_0 \in B_H(K)$, the equation \eqref{NS random} has a global solution on $[0,T]$. This assertion will follow if one can show that the set $A_i^c(\omega)= \{\xi_0 \in B_H(K): \|\xi^{N_i,n_i}_{R_K}(\cdot, \xi_0) \|_{C([0,T], H^{-\delta})} \leq R_K \}$ is dense in $B_H(K)$ in the weak topology. Indeed, it is not difficult to show that, for the deterministic equation \eqref{NS random}, the mapping $B_H(K) \ni \xi_0 \mapsto \|\xi(\cdot, \xi_0)\|_{C([0,T], H^{-\delta})}$ is continuous with respect to the weak topology on $B_H(K)$. Thus, if $A_i^c(\omega)$ is dense in $B_H(K)$, we can show that $\|\xi^{N_i,n_i}_{R_K}(\cdot, \xi_0) \|_{C([0,T], H^{-\delta})} \leq R_K$ for all $\xi_0\in B_H(K)$.
\end{remark}

\section{Wong--Zakai result: the proof of Theorem \ref{wong zakai in law}} \label{sec-Wong-Zakai}

Throughout this section, the parameter $N$ is kept fixed and  omitted
for notational simplicity. In order to simplify the notations, we will also modify the
sub/superscrips from the notations of the previous section.

Fix a finite dimensional Brownian motion $W = (W^{k, \alpha})_{k, \alpha}$,
with $k \in \tmop{supp} \theta^N$, $\alpha = 1, 2$, on a probability space
$(\Omega, \mathcal{F}, \mathbb{P})$ and let $(\mathcal{F}_t)_{t \geqslant 0}$
be its normal filtration. Let $W^n = (W^{k, \alpha, n})_{k, \alpha}$ be its
piecewise linear approximation based on a sequence of partitions $(\pi^n)_{n
\in\bN}$ of the interval $[0, T]$ with vanishing mesh size $h^n=O(\frac1n)$. In other words, the sample paths of $W^n$ are of bounded
variation and the noise term in {\eqref{random PDE with cut-off}} is given by
the classical Riemann integral. We note that the piecewise linear
approximation is not adapted to $(\mathcal{F}_t)_{t \geqslant 0}$, but it is adapted to $(\mathcal{F}_{t + h^n})_{t \geqslant 0}$.

The proof of Theorem \ref{wong zakai in law} relies on the framework of
unbounded rough drivers as developed in {\cite{BaGu15}}, {\cite{DeGuHoTi16}}
and in the context of the Navier--Stokes system in {\cite{HLN}},
{\cite{HLN19}}. To be more precise, recalling that due to the cut-off
$\theta^N$, the considered noise is finite dimensional, integrating
{\eqref{random PDE with cut-off}} in time over an interval $[s, t] \subset [0,
T]$ and iterating the equation into itself we may rewrite {\eqref{random PDE
with cut-off}} as
\begin{equation}
  \label{expansion equation} \delta \xi_{st}^{R, n} = \int_s^t [\Delta
  \xi_r^{R, n} - f_R (\| \xi_r^{R, n} \|_{- \delta})\mathcal{L}_{u_r^{R, n}}
  \xi_r^{R, n}] dr + A_{st}^{n, 1} \xi_s^{R, n} + A_{st}^{n, 2} \xi_s^{R, n} +
  \xi_{st}^{R, n, \natural}
\end{equation}
with unbounded rough drivers
\begin{equation}
  A_{st}^{n, 1} \phi = \frac{C_{\nu}}{\| \theta^N \|_{\ell^2}}  \sum_{k \in
  \mathbb{Z}_0^3} \sum_{\alpha = 1}^2 \theta^N_k \Pi (\sigma_{k, \alpha} \cdot
  \nabla \phi) \delta W_{st}^{k, \alpha, n} \label{eq:RD1}
\end{equation}
\begin{equation}
  A_{st}^{n, 2} \phi = \frac{C^2_{\nu}}{\| \theta^N \|_{\ell^2}^2}  \sum_{k, l
  \in \mathbb{Z}_0^3} \sum_{\alpha, \beta = 1}^2 \theta^N_k \Pi (\sigma_{k,
  \alpha} \cdot \nabla [\theta^N_\ell \Pi (\sigma_{\ell, \beta} \cdot \nabla \phi)])
  \mathbb{W}_{st}^{\ell, k, \beta, \alpha, n} \label{eq:RD2}
\end{equation}
where
\[ \mathbb{W}_{st}^{\ell, k, \beta, \alpha, n} = \int_s^t \delta W_{sr}^{\ell,
   \beta, n}  \dot{W}_r^{k, \alpha, n} dr. \]
A detailed discussion of this step can be found in Section 2.5 in
{\cite{HLN}}. We recall that the term $\xi^{R, n, \natural}$ is defined
through {\eqref{expansion equation}} and shall be a remainder in the sense
that it has sufficient time regularity, namely, $\xi^{R, n, \natural} \in
C^{3\alpha}_{2,\rm{loc}} ([0, T] ; H^{- 3})$.

According to Exercise 10.14 in \cite{FrHa14},
the approximate rough path
$(W^n, \mathbb{W}^n)$ converges to $(W, \mathbb{W})$ in the rough path
topology $\PP$-a.s. and in every $L^q (\Omega)$ for $q \in [1, \infty)$. Consequently, we deduce the
convergence of the associated unbounded rough drivers $(A^{n, 1}, A^{n, 2})$ to the limit
unbounded rough driver given by
\begin{equation}
  A_{st}^1 \phi = \frac{C_{\nu}}{\| \theta^N \|_{\ell^2}}  \sum_{k \in
  \mathbb{Z}_0^3} \sum_{\alpha = 1}^2 \theta^N_k \Pi (\sigma_{k, \alpha} \cdot
  \nabla \phi) \delta W_{st}^{k, \alpha} \label{eq:RD11}
\end{equation}
\begin{equation}
  A_{st}^2 \phi = \frac{C^2_{\nu}}{\| \theta^N \|_{\ell^2}^2}  \sum_{k, l \in
  \mathbb{Z}_0^3} \sum_{\alpha, \beta = 1}^2 \theta^N_k \Pi (\sigma_{k,
  \alpha} \cdot \nabla [\theta^N_\ell \Pi (\sigma_{\ell, \beta} \cdot \nabla \phi)])
  \mathbb{W}_{st}^{\ell, k, \beta, \alpha} \label{eq:RD22}
\end{equation}
where the associated rough path $(W, \mathbb{W})$ corresponds to the
Stratonovich lift, i.e.,
\[ \mathbb{W}_{st}^{\ell, k, \beta, \alpha} = \int_s^t \delta W_{sr}^{\ell, \beta}
   \circ dW_r^{k, \alpha} . \]
This means that the operators $(A^{n, 1}, A^{n, 2})$ satisfy the bounds
\begin{equation*}
  \|A_{st}^{n, 1} \|_{\mathcal{L} (H^{-k} ; H^{-(k+1)})} \lesssim C_{A^{n}} |t-s|^{\alpha}, \qquad \|A_{st}^{n, 2} \|_{\mathcal{L} (H^{-k} ; H^{-(k+2)})}
  \le C_{A^n} |t-s|^{2\alpha},
\end{equation*}
where the first bound holds for $k\in \{0,2\}$ whereas the second one for $k\in\{0,1\}$. In addition, by Exercise 10.14 in \cite{FrHa14} we have for all $q \in [1,
\infty)$
\begin{equation}
  \sup_{n\in\bN} \mathbb{E} [C_{A^{n}}^q] < \infty  \label{eq:13}
\end{equation}
and $\PP$-a.s.
\begin{equation}\label{eq:14}
\sup_{n\in\bN} C_{A^{n}}(\omega)\le C(\omega)
\end{equation}
for some random constant $C(\omega)$.
Accordingly, the rough path formulation of (\ref{SPDE with cut-off}) reads
as
\begin{equation}
  \delta \xi_{st}^R = \int_s^t [\Delta \xi_r^R - f_R (\| \xi_r^R \|_{-
  \delta})\mathcal{L}_{u_r^R} \xi_r^R] dr + A_{st}^1 \xi_s^R + A_{st}^2
  \xi_s^R + \xi_{st}^{R, \natural} . \label{eq:rp3}
\end{equation}
In view of Corollary 5.2 in {\cite{FrHa14}}, we aim to conclude that an
adapted rough path solution to (\ref{SPDE with cut-off}) is also a solution in
the classical (stochastic) sense. Since the adaptedness is the key point
needed for the construction of the stochastic integral, we have to make sure
that our Wong--Zakai convergence result produces adapted solutions. By merely
pathwise arguments we are not able to construct adapted solutions to
{\eqref{eq:rp3}}. The principal difficulty is that we are not able to prove
uniqueness of rough path solutions to {\eqref{eq:rp3}} and therefore a
pathwise compactness argument does not preserve measurability in $\omega$. To
overcome this obstacle, we combine rough path techniques together with
probabilistic arguments, namely, the stochastic compactness method based on
Skorokhod representation theorem. This permits to make use of the uniqueness
for the stochastic version of {\eqref{eq:rp3}}, i.e., the equation (\ref{SPDE
with cut-off}), and eventually construct adapted solutions to
{\eqref{eq:rp3}}.

An additional technical difficulty follows from the fact that the
approximation $W^n$ is not adapted to $(\mathcal{F}_t)_{t \geqslant 0}$ but
only to $(\mathcal{F}_{t + h^n})_{t \geqslant 0}$. While this point can be
fixed for instance by replacing piecewise linear approximations by one-sided
mollifications which remain adapted, we choose to work with piecewise linear
approximations as they are better suited for applications in numerical
analysis.

As the first step, we establish the necessary uniform estimates.

\begin{proposition}
  \label{prop:apriori}
  There exists a unique solution $\xi^{R, n}$ to
  {\eqref{expansion equation}} and it is adapted to $(\mathcal{F}_{t +
  h^n})_{t \geqslant 0}$. Moreover, it holds
    \[ \| \xi^{R, n} \|^2_{L^{\infty}_{T} H} + \| \xi^{R, n} \|^2_{L^2_{T}
     H^1} \leqslant C (R, K), \]
      \[ \| \xi^{R, n} \|_{C^{\alpha}_{T}H^{- 1}} \lesssim (1 + C
     (R, K)) (1+C_{A^{n}}^{2}), \]
  and   there exists a  deterministic constant $L > 0$ such that whenever $C_{A^{n}}|t-s|^{3\alpha}\leq L$ we have
  \[ \| \xi_{s t}^{R, n, \natural} \|_{ H^{- 3}} \lesssim (1 + C (R, K))(1+C_{A^{n}}^{2})|t-s|^{3\alpha}, \]
  for some deterministic implicit constant independent of $n$.
\end{proposition}

\begin{proof}
  Existence and uniqueness of a solution $\xi^{R, n}$ follows by classical
  arguments since the driver $(A^{n, 1}, A^{n, 2})$ is smooth. Since $W^n$ is
  adapted to $(\mathcal{F}_{t + h^n})_{t \geqslant 0}$, the same remains valid
  for the solution $\xi^{R, n}$.

  The first bound in the statement of the proposition follows from
  {\eqref{random PDE with cut-off energy bound}}. For the other two estimates,
  we intend to apply Theorem \ref{Thm2.5Abstract}. Thus, we shall derive the necessary bounds for the drift term
  $$
  \delta \mu_{s t} = \int_s^t [\Delta \xi_r^{R, n} - f_R
     (\| \xi_r^{R, n} \|_{H^{- \delta}})\mathcal{L}_{u_r^{R, n}} \xi_r^{R, n}] dr
  $$
  in $H^{-2}$ and $H^{-1}$.
To this end, we observe that since $u^{R, n}$ is divergence free we have
  \[\aligned \langle \mathcal{L}_{u_r^{R, n}} \xi_r^{R, n}, \phi \rangle &= \langle
     (u_r^{R, n} \cdummy \nabla) \xi_r^{R, n}, \phi \rangle - \langle
     (\xi_r^{R, n} \cdummy \nabla) u_r^{R, n}, \phi \rangle \\
     &= - \langle (u_r^{R, n} \cdummy \nabla) \phi, \xi_r^{R, n} \rangle -
     \langle (\xi_r^{R, n} \cdummy \nabla) u_r^{R, n}, \phi \rangle .
     \endaligned \]
  Hence in view of {\eqref{trilinear form estimate}} with $m_1 = m_2 =
  1$, $m_3 = 0$ for the first term and $m_1= m_2 = 0$, $m_3 = 2$ for the
  second term we obtain
  \[ \| \mathcal{L}_{u_r^{R, n}} \xi_r^{R, n} \|_{H^{- 2}} \lesssim \| \xi_r^{R,
     n} \|_H (1 + \| u^{R, n}_r \|_{H^1}) . \]
  Therefore due to {\eqref{random PDE with cut-off energy bound}}
  \[ \|\delta{\mu}_{s t}\|_{H^{-2}} \lesssim \int_s^t \| \xi_r^{R, n} \|_H (1 + \| u^{R,
     n}_r \|_{H^1}) d r \lesssim (t - s) (1 + C (R, K)), \]
  with a deterministic implicit constant and Theorem~\ref{Thm2.5Abstract} implies
  \[ \| \xi^{R, n, \natural}_{s t} \|_{H^{- 3}} \lesssim (1 + C (R, K))(1+C_{A^{n}}^{2})|t-s|^{3\alpha}, \]
  which gives the desired bound of the remainder.

  Finally, we observe that by {\eqref{trilinear form estimate}} the drift can
  be estimated in $H^{- 1}$ as follows
  \[\aligned \|\delta\mu_{s t}\|_{H^{- 1}}
     & \lesssim \int_s^t \| \xi^{R, n}_r \|_{H^1} (1 + \| u^{R, n}_r \|_{H^1}) d
     r \\
  & \lesssim (t - s)^{\frac{1}{2}} | \xi^{R, n} |_{L^2_{T} H^1} (1 + |
     u^{R, n} |_{L^{\infty}_{T} H^1}) \lesssim |t - s|^{\frac{1}{2}} (1 +
     C (R, K)) .  \endaligned \]
  Hence Theorem \ref{Thm2.5Abstract} implies
  \[ |\xi^{R, n}|_{C^{\alpha}_{T}H^{- 1}} \lesssim (1 + C
     (R, K)) (1+C^{2}_{A^{n}})\]
  and the proof is complete.
\end{proof}

Now, we have all in hand to prove the following result (recall that we fix $N\in \mathbb Z_+$ in this section).

\begin{proposition}\label{prop-intermediate}
Let $\{\xi^n_0 \}_{n\geq 1} \subset H$ be a sequence satisfying $\|\xi^n_0\|_H \leq K$ for all $n\geq 1$, and $\xi^{R,n}$ the unique solution to \eqref{random PDE with cut-off} with $\xi^{R, n}(0)= \xi^n_0$. Assume that $\xi^n_0$ converges weakly in $H$ to some $\xi_0$ as $n\to \infty$; then $\xi^{R,n}$ converge in probability in the topology of $C ([0, T] ; H^{- \delta})$ to $\xi^{R}$, the solution to \eqref{SPDE with cut-off} with initial value $\xi_0$.
\end{proposition}

\begin{proof}
  \tmtextit{Step 1: Tightness.} We define the space
  \[ \mathcal{X} \assign \mathcal{X}_{\xi} \times \mathcal{X}_{\tmop{RP}}, \]
  \[ \mathcal{X}_{\xi} \assign L^2 (0, T ; H) \cap C ([0, T] ; H^{- \delta}),
     \qquad \mathcal{X}_{\tmop{RP}} \assign C_{2}^{\alpha} ([0, T] ;
     \bR^m) \times C^{2\alpha}_2 ([0, T] ;
     \bR^{m\times m}) , \]
     where $m\in\bN$ is the dimension of the Brownian motion $W$.
  Next, we claim  that $L^{\infty} (0, T ; H) \cap L^2 (0, T ; H^1) \cap
  C^{\alpha} ([0, T] ; H^{- 1})$  is compactly embedded into
  $\mathcal{X}_{\xi}$. Indeed, the compactness of the embedding into $L^{2}(0,T;H)$ follows from \cite[Corollary~1.8.4]{BrFeHobook}. On the other hand, by \cite[Theorem~1.8.5]{BrFeHobook} we obtain a compact embedding into the space of weakly continuous functions $C_{\rm{weak}}([0,T];H)$, which is continuously embedded into $C([0,T];H^{-\delta})$ for any $\delta>0$.

%
%
%
%


  Consequently, due to Proposition \ref{prop:apriori} and
  in particular due to the fact that the right hand sides of the estimates are
  uniformly bounded in expectation due to {\eqref{eq:13}}, the family of the
  pushforward measures $(\xi^{R,n})_{\sharp} \PP$ is tight on $\mathcal{X}_{\xi}$.
  In order to apply the theory of rough paths for the passage limit we shall
  also need the structure of the noise. Since due to Exercise 10.14 in {\cite{FrHa14}}, the
  approximate rough path $(W^n, \mathbb{W}^n)$ converges $\mathbb{P}$-a.s. in
  the rough path topology to the Stratonovich lift of a Brownian motion $W$,
  the family of joint laws of $(W^n, \mathbb{W}^n)$ is tight on
  $\mathcal{X}_{\tmop{RP}}$ which is separable. Thus, we deduce that the joint
  laws $(\xi^{R,n}, W^n, \mathbb{W}^n)_{\sharp} \PP$ are tight as a family of
  probability measures on $\mathcal{X}$.

  From the Skorokhod representation
  theorem which applies to Polish spaces (see e.g. Section~2.6 in {\cite{BrFeHobook}}) there exists a
  probability space $(\tilde{\Omega}, \tilde{\mathcal{F}},
  \tilde{\mathbb{P}})$ and random variables
  \[(\tilde{\xi}^{R, n}, \tilde{W}^n, \tilde{\mathbb{W}}^n) :
     \tilde{\Omega} \rightarrow \mathcal{X},\quad n\in\bN, \qquad
     (\tilde{\xi}^R, \tilde{W}, \tilde{\mathbb{W}}) : \tilde{\Omega}
     \rightarrow \mathcal{X}, \]
  such that (up to a subsequence)
  \begin{enumerate}
    \item $(\tilde{\xi}^{R, n}, \tilde{W}^n, \tilde{\mathbb{W}}^n)  \rightarrow (\tilde{\xi}^R, \tilde{W}, \tilde{\mathbb{W}})$  in $\mathcal{X}$ $\tilde{\PP}$-a.s.
    as $n \rightarrow \infty$,

    \item $(\tilde{\xi}^{R, n}, \tilde{W}^n, \tilde{\mathbb{W}}^n)_{\sharp} \tilde{\PP} = (\xi^{R,n}, W^n,
    \mathbb{W}^n)_{\sharp} \PP $ for all $n\in\bN$.
  \end{enumerate}
  We  define $(\tilde{\mathcal{F}}_t)_{t \geqslant 0}$ as the
  augmented canonical filtration generated by $     (\tilde{\xi}^R, \tilde{W}, \tilde{\mathbb{W}})$, that is, we let
  \[ \tilde{\mathcal{F}}_t \assign \sigma (\sigma ( \tilde{\xi}^{R}_s,\tilde{W}_{rs} ,
     \tilde{\mathbb{W}}_{r s} ; 0 \leqslant r \leqslant s \leqslant t) \cup \{
     N ; \tilde{\mathbb{P}} (N) = 0 \}), \qquad t \geqslant 0. \]

  \tmtextit{Step 2: Passage to the limit.} As the next step, we shall prove
  that $(\tilde{\xi}^{R, n}, \tilde{W}^n, \tilde{\mathbb{W}}^n)$ gives rise to a solution of {\eqref{expansion equation}}
  on the new probability space. First, we shall identify the corresponding
  rough path. To this end, we observe that Chen's relation giving the
  necessary compatibility condition between components of a rough path holds
  for $(\tilde{W}^n, \tilde{\mathbb{W}}^n)$ as well. Indeed, it follows from
  the equality of laws
  \[\aligned &\ \tilde{\mathbb{P}} \left( \delta \tilde{\mathbb{W}}_{r s t}^n =
     \tilde{W}_{r s}^n \otimes \tilde{W}_{s t}^n \quad \mbox{for all }
     \quad 0 \leqslant r \leqslant s \leqslant t \leqslant T \right) \\
  =&\ \mathbb{P} \left( \delta \mathbb{W}_{r s t}^n = W_{r s}^n \otimes W_{s
     t}^n \quad \mbox{for all } \quad 0 \leqslant r \leqslant s
     \leqslant t \leqslant T \right) = 1. \endaligned \]
  In other words, $(\tilde{W}^n, \tilde{\mathbb{W}}^n)$ is a well-defined
  rough path $\tilde{\mathbb{P}}$-a.s. Hence we may define the unbounded rough
  drivers $(\tilde{A}^{n, 1}, \tilde{A}^{n, 2})$ through the formulas
  {\eqref{eq:RD1}}, {\eqref{eq:RD2}} with $(W^n, \mathbb{W}^n)$ replaced by
  $(\tilde{W}^n, \tilde{\mathbb{W}}^n)$.

  Let us now define
  \begin{equation}
    \tilde{\xi}_{st}^{R, n, \natural} \assign \delta \tilde{\xi}_{st}^{R, n} -
    \int_s^t [\Delta \tilde{\xi}_r^{R, n} - f_R (\| \tilde{\xi}_r^{R, n} \|_{H^{-
    \delta}})\mathcal{L}_{u_r^{R, n}}  \tilde{\xi}_r^{R, n} ] dr -
    \tilde{A}_{st}^{n, 1}  \tilde{\xi}_s^{R, n} - \tilde{A}_{st}^{n, 2}
    \tilde{\xi}_s^{R, n} . \label{eq:rem}
  \end{equation}
  Recall that {\eqref{expansion equation}} was satisfied on the original probability
  space  even in the classical formulation \eqref{random PDE with cut-off} as the driving path $W^{n}$ is regular. In addition, the right hand side of {\eqref{eq:rem}} is a measurable function
  of $(\tilde{\xi}^{R, n}, \tilde{W}^n, \tilde{\mathbb{W}}^n)$, we deduce again by  the equality of joint laws
that $\tilde{\xi}^{R, n}$ is $\tilde{\mathbb{P}}$-a.s. a solution
  to
  \begin{equation}
    \label{expansion equation for tilde n} \delta \tilde{\xi}_{st}^{R, n} =
    \int_s^t [\Delta \tilde{\xi}_r^{R, n} - f_R (\| \tilde{\xi}_r^{R, n} \|_{H^{-
    \delta}})\mathcal{L}_{u_r^{R, n}}  \tilde{\xi}_r^{R, n}] dr +
    \tilde{A}_{st}^{n, 1}  \tilde{\xi}_s^{R, n} + \tilde{A}_{st}^{n, 2}
    \tilde{\xi}_s^{R, n} + \tilde{\xi}_{st}^{R, n, \natural},
  \end{equation}
  which is the rough path formulation of {\eqref{random PDE with cut-off}} on
  the probability space $(\tilde{\Omega}, \tilde{\mathcal{F}},
  \tilde{\mathbb{P}})$.

  Our goal is to pass to the limit in {\eqref{expansion equation for tilde n}}
  and identify the limit $\tilde{\xi}^R$ as a solution to {\eqref{eq:rp3}}
  where the rough driver is given by the Stratonovich lift of a Brownian
  motion. To this end, we note that the approach of Proposition
  \ref{prop:apriori} giving the uniform bounds can be applied to
  {\eqref{expansion equation for tilde n}} as well. In particular, in view of \eqref{eq:14} we obtain a pathwise uniform bound for the remainders $\tilde\xi^{R,n,\natural}$.
  Together with the $\tilde{\mathbb{P}}$-a.s.
  convergence of $(\tilde{W}^n, \tilde{\mathbb{W}}^n) \rightarrow (\tilde{W},
  \tilde{\mathbb{W}})$ in $\mathcal{X}_{\tmop{RP}}$, this permits to pass to
  the limit in {\eqref{expansion equation for tilde n}}. Note in particular
  that  passing to the limit in the Chen's relation guarantees that the
  limit $(\tilde{W}, \tilde{\mathbb{W}})$ is a rough path itself. Therefore,
  we obtain the convergence of the rough driver $(\tilde{A}^{n, 1},
  \tilde{A}^{n, 2})$ to $(\tilde{A}^1, \tilde{A}^2)$ given by
  {\eqref{eq:RD11}}, {\eqref{eq:RD22}} with $(W, \mathbb{W})$ replaced by
  $(\tilde{W}, \tilde{\mathbb{W}})$. Hence the limit satisfies
  \begin{equation}
    \label{expansion equation for tilde} \delta \tilde{\xi}_{st}^R = \int_s^t
    [\Delta \tilde{\xi}_r^R - f_R (\| \tilde{\xi}_r^R \|_{H^{-
    \delta}})\mathcal{L}_{u_r^R}  \tilde{\xi}_r^R] dr + \tilde{A}_{st}^1
    \tilde{\xi}_s^R + \tilde{A}_{st}^2  \tilde{\xi}_s^R + \tilde{\xi}_{st}^{R,
    \natural}
  \end{equation}
  for some remainder $\tilde{\xi}^{R, \natural}$ which belongs $\tilde\PP$-a.s. to $C^{3\alpha}_{2,\rm{loc}}([0,T];H^{-3})$.

  \tmtextit{Step 3: Identification of the limiting driver.} We have shown that
  the limit $\tilde{\xi}^R$ solves the rough path formulation of {\eqref{SPDE
  with cut-off}} and it only remains to prove that it is also a solution of
  {\eqref{SPDE with cut-off}} in the classical stochastic sense. To this end,
  it is necessary to identify $(\tilde{W}, \tilde{\mathbb{W}})$ as the
  Stratonovich lift of a Brownian motion.

  Since $(\tilde{W}^n, \tilde{\mathbb{W}}^n)$ is equal in law to $(W^n,
  \mathbb{W}^n)$ which converges $\mathbb{P}$-a.s. to the Stratonovich lift of
  the Brownian motion $W$, we deduce that $(\tilde{W}^n,
  \tilde{\mathbb{W}}^n)$ converges in law to $(W, \mathbb{W})$. As a
  consequence, $\tilde{W}$ has the same law as $W$ and therefore it is an increment of a Brownian motion. Next, we show that it
  is a Brownian motion with respect to $(\tilde{\mathcal{F}}_t)_{t
  \geqslant 0}$. To this end, fix arbitrary times
   $0 \leqslant r \leqslant s <
  t \leqslant T$ and an arbitrary continuous function $\gamma : C ([0, s] ;
  H^{- \delta}) \times  C^{0}_{2} ([r, s] ;
  \bR^m) \times C^{0}_2 ([r, s] ; \bR^{m \times m}) \rightarrow [0,
  1]$. Due to equality of joint laws it holds
  \[\aligned \tilde{\mathbb{E}} [\gamma (\tilde{\xi}^{R}|_{[0, s]} \nobracket, \tilde W|_{[r,s]},
     \tilde{\mathbb{W}} |_{[r, s]} \nobracket) \tilde{W}_{st}] &=
     \lim_{n \rightarrow \infty} \tilde{\mathbb{E}} [\gamma (
     \tilde{\xi}^{R,n} |_{[0, s]} \nobracket,\tilde{W}^{n}|_{[r, s]}, \tilde{\mathbb{W}}^n |_{[r, s]}
     \nobracket) \tilde{W}^n_{st}] \\
  &= \lim_{n \rightarrow \infty} \mathbb{E} [\gamma (\xi^{R,n} |_{[0, s]}
     \nobracket, {W}^{n}|_{[r, s]},\mathbb{W}^n |_{[r, s]} \nobracket) W^n_{st}] . \endaligned \]
  Since for every $n \in \bN$ the random variable $(\xi^{R,n} |_{[0, s]}
     \nobracket, {W}^{n}|_{[r, s]},\mathbb{W}^n |_{[r, s]} \nobracket)$ is measurable with
  respect to the $\sigma$-algebra $\mathcal{F}_{s + h^n}$, it can be written
  as a measurable function of $W |_{[0, s + h^n]} \nobracket$, say
  $$(\xi^{R,n} |_{[0, s]}
     \nobracket, {W}^{n}|_{[r, s]},\mathbb{W}^n |_{[r, s]} \nobracket) = F^n (W
  |_{[0, s + h^n]} \nobracket).$$
  Consequently,
  \begin{equation}
    \tilde{\mathbb{E}} [\gamma (\tilde{\xi}^{R}|_{[0, s]} \nobracket, \tilde W|_{[r,s]},
     \tilde{\mathbb{W}} |_{[r, s]} \nobracket) \tilde{W}_{st}] =
    \lim_{n \rightarrow \infty} \mathbb{E} [\gamma (F^n (W |_{[0, s + h^n]}
    \nobracket)) W^n_{st}] . \label{eq:11}
  \end{equation}
  Since the function $\gamma$ is bounded, the sequence $\gamma (F^n (W |_{[0,
  s + h^n]} \nobracket))$ is uniformly bounded in $n$. Hence there is a
  subsequence converging weak star in $L^{\infty} (\Omega)$. In addition, the
  limit denoted by $\Gamma_s$ is $\mathcal{F}_{s +}$-measurable since for
  every $h \in (0, 1)$ it is a weak star limit of $\mathcal{F}_{s +
  h}$-measurable functions, i.e., the weak star limit can be taken in
  $L^{\infty} (\Omega, \mathcal{F}_{s + h}, \mathbb{P})$. Due to right
  continuity of the filtration $(\mathcal{F}_t)_{t \geqslant 0}$ it follows
  that $\Gamma_s$ is $\mathcal{F}_s$-measurable.

  On the other hand, since $W^n$
  converges in every $L^q (\Omega)$, and thus is bounded therein uniformly in $n$, we obtain from \eqref{eq:11} by weak-strong convergence
  \[ \tilde{\mathbb{E}} [\gamma (\tilde{\xi}^{R}|_{[0, s]} \nobracket, \tilde W|_{[r,s]},
     \tilde{\mathbb{W}} |_{[r, s]} \nobracket) \tilde{W}_{st}]
     =\mathbb{E} [\Gamma_s  W_{st}] = 0, \]
  where the last equality follows from the martingale property of $W$ with
  respect to $(\mathcal{F}_t)_{t \geqslant 0}$. This shows that $t\mapsto\tilde{W}_{0t}$ is
  a $(\tilde{\mathcal{F}}_t)_{t \geqslant 0}$-martingale and hence a
  $(\tilde{\mathcal{F}}_t)_{t \geqslant 0}$-Brownian motion.

  It remains to identify $\tilde{\mathbb{W}}$ as the Stratonovich lift of
  $\tilde{W}$.
  More precisely, we want to prove that
  \[ \tilde{\mathbb{W}}^{\ell, k, \beta, \alpha}_{s t} = \int_s^t
     \tilde{W}^{\ell, \beta}_r \circ d \tilde{W}^{k, \alpha}_r -
     \tilde{W}^{\ell, \beta}_s \tilde{W}^{k, \alpha}_{s t} \]
  holds $\tilde{\mathbb{P}}$-a.s. for all $0 \leqslant s \leqslant t \leqslant
  T$.
  The right hand side can be rewritten in terms of an It{\^o} integral and the
  corresponding cross variation as follows
  \[ \tilde{\mathbb{W}}^{\ell, k, \beta, \alpha}_{s t} = \int_s^t
     \tilde{W}^{\ell, \beta}_r d \tilde{W}^{k, \alpha}_r + \frac{1}{2} \delta
     \llangle \tilde{W}^{\ell, \beta}, \tilde{W}^{k, \alpha} \rrangle_{s t} -
     \tilde{W}^{\ell, \beta}_s \tilde{W}^{k, \alpha}_{s t} = \int_s^t \tilde{W}^{\ell, \beta}_{s r} d \tilde{W}^{k, \alpha}_r +
     \frac{1}{2} (t - s) \delta_{k = \ell, \alpha = \beta} .\]
  In other words, regarding $s$ as an initial time the above says that $t
  \mapsto \tilde{\mathbb{W}}^{\ell, k, \beta, \alpha}_{s t}$ should solve an
  It{\^o} stochastic differential equation. Let us define the process
  \[ t \mapsto \tilde{M}_t \assign \tilde{\mathbb{W}}^{\ell, k, \beta,
     \alpha}_{s t} - \frac{1}{2} (t - s) \delta_{k = \ell, \alpha = \beta} .
  \]
  Once we prove that
  \begin{equation}
    \tilde{M}_t = \int_s^t \tilde{W}^{\ell, \beta}_{s r} d \tilde{W}^{k,
    \alpha}_r, \label{eq:12}
  \end{equation}
  holds $\tilde{\mathbb{P}}$-a.s. for all $0 \leqslant s \leqslant t \leqslant
  T$, the identification of $\tilde{\mathbb{W}}$ is complete.

  To this end, since we already know that $(\tilde W,\tilde{\mathbb{W}})$ equals in law to the
  Stratonovich lift $(W,\WW)$, we define analogously on the original probability space
  \[t\mapsto M_t \assign \mathbb{W}^{\ell, k, \beta, \alpha}_{s t} - \frac{1}{2}
     (t - s) \delta_{k = \ell, \alpha = \beta}\]
     and here we know that $\PP$-a.s.
     \begin{equation}\label{eq:15}
     M_t=\int_{s}^{t}W^{\ell,\beta}_{s r}d W^{k,\alpha}_{r}.
     \end{equation}
 We will use martingale arguments to deduce \eqref{eq:12} from \eqref{eq:15} and from the equality of joint laws. For $0 \leqslant s \leqslant \tau \leqslant \sigma < t \leqslant T$ and an arbitrary continuous function $\gamma :   C^{0}_{2} ([\tau, \sigma] ;
  \bR^m) \times C^{0}_2 ([\tau, \sigma] ; \bR^{m \times m}) \rightarrow [0,
  1]$,
  we obtain from the equality
  of joint laws of $(\tilde W,\tilde\WW)$ and $(W,\WW)$
  \[ \tilde{\mathbb{E}} [\gamma ( \tilde{W}|_{{[\tau,\sigma]}},\tilde{\mathbb{W}} |_{[\tau, \sigma]} \nobracket)
     (\tilde{M}_t - \tilde{M}_{\sigma})] =
     {\mathbb{E}} [\gamma ( {W}|_{{[\tau,\sigma]}},{\mathbb{W}} |_{[\tau, \sigma]} \nobracket)
     ({M}_t - {M}_{\sigma})]. \]
The right hand side vanishes due to \eqref{eq:15} and accordingly,
$\tilde{M}$ is a martingale with respect to the filtration generated by $(\tilde W,\tilde\WW)$.  In order to deduce that $\tilde{M}$ is equal to the stochastic integral
  {\eqref{eq:12}}, it remains to identify its quadratic variation as the cross variation with the driving
  process $\tilde{W}^{k, \alpha}$. Proceeding by the same arguments we deduce
  \[ \tilde{\mathbb{E}} \left[ \gamma ( \tilde{W}|_{{[\tau,\sigma]}},\tilde{\mathbb{W}} |_{[\tau, \sigma]} \nobracket) \left(
     \tilde{M}_t^2 - \tilde{M}_{\sigma}^2 - \int_{\sigma}^t (\tilde{W}^{\ell,
     \beta}_{s r})^2 d r \right) \right] = 0,\]
      \[ \tilde{\mathbb{E}} \left[ \gamma ( \tilde{W}|_{{[\tau,\sigma]}},\tilde{\mathbb{W}} |_{[\tau, \sigma]} \nobracket) \left(
     \tilde{M}_t \tilde{W}^{k, \alpha}_t - \tilde{M}_{\sigma} \tilde{W}^{k,
     \alpha}_{\sigma} - \int_{\sigma}^t \tilde{W}^{\ell, \beta}_{s r} d r
     \right) \right] = 0, \]
in other words,
  \begin{equation*}
    \llangle \tilde{M} \rrangle_t = \int_s^t (\tilde{W}^{\ell, \beta}_{s r})^2dr,\qquad \llangle \tilde{M}, \tilde{W}^{k, \alpha}_{} \rrangle_t = \int_s^t
     \tilde{W}^{\ell, \beta}_{s r} d r.
  \end{equation*}
Therefore,
  \[ \llangle \tilde{M} - \int_s^. W^{\ell, \beta}_{s r} d W^{k, \alpha}_r
     \rrangle_t = \llangle \tilde{M} \rrangle_t - 2 \llangle \tilde{M},
     \int_s^. W^{\ell, \beta}_{s r} d W^{k, \alpha}_r \rrangle_t + \llangle
     \int_s^. W^{\ell, \beta}_{s r} d W^{k, \alpha}_r \rrangle_t = 0, \]
  which finally implies that for every $0 \leqslant s \leqslant T$ the
  equality {\eqref{eq:12}} holds $\tilde{\mathbb{P}}$-a.s. for all $s
  \leqslant t \leqslant T$. So far the associated set of full probability
  depends on $s$, however, by continuity of the involved quantities in $s$ the
  desired result follows.

  Therefore, we have proved that $(\tilde{W}^n, \tilde{\WW}^n)$ converges
  to $(\tilde{W}, \tilde{\WW})$ in the rough path topology
  $\tilde{\mathbb{P}}$-a.s., where the iterated integral $\tilde{\mathbb{W}}$
  is the Stratonovich lift of the Brownian motion $\tilde{W}$. This permits to
  identify $\tilde{\xi}^R$ as a solution to the stochastic equation
  {\eqref{SPDE with cut-off}}. Indeed, fix a test function $\phi \in
  C^{\infty} \cap H$ and $\tilde{\omega} \in \tilde{\Omega}$ from the set of
  full probability $\tilde{\mathbb{P}}$ where the convergences above hold. In
  particular, we have $(\tilde{W} (\tilde{\omega}), \tilde{\WW}
  (\tilde{\omega})) \in \mathcal{X}_{\tmop{RP}}$. From {\eqref{expansion
  equation for tilde}} we see that $(\langle \tilde{\xi}^R (\tilde{\omega}),
  \phi \rangle, \langle \tilde{\xi}^R (\tilde{\omega}), - \mathrm{div}
  (\sigma_{k, \alpha} \phi) \rangle)$ is a controlled rough path corresponding to $(\tilde{W} (\tilde{\omega}), \tilde{\WW}
  (\tilde{\omega}))$. Hence in view of the adaptedness of
  $\tilde{\xi}^R$ to $(\tilde{\mathcal{F}}_t)_{t \geqslant 0},$ it follows
  from Corollary 5.2 in {\cite{FrHa14}} that $\tilde\PP$-a.s.
  \[ \frac{C_{\nu}}{\| \theta^N \|_{\ell^2}}  \sum_{k \in \mathbb{Z}_0^3}
     \sum_{\alpha = 1}^2  \theta_k^N  \int_s^t \big\langle \tilde{\xi}^R_r,
     - \mathrm{div} (\sigma_{k, \alpha} \phi) \big\rangle \circ d \tilde{W}_r^{k,
     \alpha} (\tilde{\omega})\]
     \[
      = \big\langle \tilde{\xi}^R_r (\tilde{\omega}),
     \tilde{A}_{st}^{1, \ast}(\tilde\omega) \phi + \tilde{A}_{st}^{2, \ast}(\tilde\omega) \phi \big\rangle +
     \big\langle \tilde{\xi}^{R, \natural}_{st} (\tilde{\omega}), \phi \big\rangle
     \]
  so that $\tilde{\xi}^R$ satisfies {\eqref{SPDE with cut-off}}.

  \tmtextit{Step 4: Convergence on the original probability space.} Since
  pathwise uniqueness holds true for {\eqref{SPDE with cut-off}}, it is
  possible to deduce that the original sequence of approximate solutions
  $\xi^{R, n}$ converges in probability on the original probability space
  $(\Omega, \mathcal{F}, \mathbb{P})$. This is the classical Yamada--Watanabe
  argument, which can be established using the Gy{\"o}ngy--Krylov lemma, see
  e.g. Section 2.10 in {\cite{BrFeHobook}} and an application of this method
  in Section 5.2.6 in {\cite{BrFeHobook}}. Then, by repeating the above limiting procedure
   on the original probability space, we deduce that the limit in
  probability, denoted by $\xi^R$, solves {\eqref{SPDE with cut-off}} on
  $(\Omega, \mathcal{F}, \mathbb{P})$. In fact, contrary to the identification
  of the limit on $(\tilde{\Omega}, \tilde{\mathcal{F}},
  \tilde{\mathbb{P}})$, the arguments simplifies significantly since no
  identification of the limiting rough path is necessary.
  This
  concludes the proof of the Wong--Zakai result.
\end{proof}

Thanks to Proposition \ref{prop-intermediate}, we can finally prove Theorem \ref{wong zakai in law}.

\begin{proof}[Proof of Theorem \ref{wong zakai in law}]
We follow the idea of the proof of Theorem 1.4 in \cite{FL19} and argue by contradiction. Assume that there exist an $\varepsilon_0>0$ and a subsequence $\{n_i\}_{i\geq 1}\subset \mathbb Z_+$ such that
  $$\lim_{i\to \infty} \sup_{\xi_0\in H, \|\xi_0\|_H\leq K} \mathbb{P} \Big(\big\|\xi_{R_K}^{N_0, n_i}(\cdot,\xi_0) - \xi_{R_K}^{N_0} (\cdot,\xi_0)\big\|_{C([0,T],H^{-\delta})} >\varepsilon_0 \Big)>0.$$
Then we can find a sequence $\{\xi^{n_i}_0 \}_{i\geq 1}\subset H$ such that $\|\xi^{n_i}_0 \|_H\leq K$ for all $i\geq 1$, and (choosing a smaller $\varepsilon_0>0$ if necessary)
  \begin{equation}\label{proof-1}
  \mathbb{P} \Big(\big\|\xi_{R_K}^{N_0, n_i}(\cdot,\xi^{n_i}_0) - \xi_{R_K}^{N_0} (\cdot,\xi^{n_i}_0)\big\|_{C([0,T],H^{-\delta})} >\varepsilon_0 \Big) \geq \varepsilon_0 >0.
  \end{equation}

Since the sequence $\{\xi^{n_i}_0 \}_{i\geq 1}\subset H$ is bounded, up to a subsequence, it converges weakly to some $\xi_0\in H$. We can repeat the proof of Proposition \ref{prop-intermediate} to show that, as $i\to \infty$, the sequence $\xi_{R_K}^{N_0, n_i}(\cdot,\xi^{n_i}_0)$ converges in probability in the topology of $C([0,T];H^{-\delta})$ to the solution $\xi_{R_K}^{N_0} (\cdot,\xi_0)$ of \eqref{SPDE with cut-off} with initial value $\xi_0$. Similarly, the other sequence $\xi_{R_K}^{N_0} (\cdot,\xi^{n_i}_0)$ converges also in probability in the topology of $C([0,T];H^{-\delta})$ to $\xi_{R_K}^{N_0} (\cdot,\xi_0)$, see for instance Lemma 4.1 in \cite{FL19} (or Corollary~3.5 therein). From these results and the following simple inequality:
  $$\aligned
  &\ \mathbb{P} \Big(\big\|\xi_{R_K}^{N_0, n_i}(\cdot,\xi^{n_i}_0) - \xi_{R_K}^{N_0} (\cdot,\xi^{n_i}_0)\big\|_{C([0,T],H^{-\delta})} >\varepsilon_0 \Big) \\
  \leq &\ \mathbb{P} \Big(\big\|\xi_{R_K}^{N_0, n_i}(\cdot,\xi^{n_i}_0) - \xi_{R_K}^{N_0} (\cdot,\xi_0) \big\|_{C([0,T], H^{-\delta})} >\frac{\varepsilon_0}2 \Big) + \mathbb{P} \Big(\big\|\xi_{R_K}^{N_0}(\cdot,\xi^{n_i}_0) - \xi_{R_K}^{N_0} (\cdot,\xi_0) \big\|_{C([0,T], H^{-\delta})} >\frac{\varepsilon_0}2 \Big),
  \endaligned $$
we immediately get a contradiction with \eqref{proof-1}. Thus we complete the proof of Theorem \ref{wong zakai in law}.
\end{proof}

\bibliographystyle{alpha}
\bibliography{bibliography}

\end{document}